\documentclass[11pt]{amsart}

\usepackage[a4paper,left=28mm,right=28mm,top=27mm,bottom=29mm]{geometry}
\usepackage[T1]{fontenc}
\usepackage{lmodern}
\usepackage{microtype}
\usepackage{mathtools,amssymb,amsthm}
\usepackage{xcolor}
\usepackage[colorlinks=true,linkcolor=blue!55!black,
            citecolor=blue!55!black,urlcolor=blue!65!black]{hyperref}
\hypersetup{
  pdftitle={Convergence and   Semisimplicity of Big Quantum Cohomology of Flag Varieties},
  pdfauthor={Changzheng Li, Zhaoyang Liu, Tony Yue YU, and Chi Zhang},
  pdfsubject={Quantum cohomology, analytic reconstruction, and flag varieties},
  pdfkeywords={big quantum cohomology, flag varieties, convergence, generic semisimplicity, quantum D-module}
}

\numberwithin{equation}{section}

\newtheorem{theorem}{Theorem}[section]
\newtheorem{proposition}[theorem]{Proposition}
\newtheorem{lemma}[theorem]{Lemma}
\newtheorem{corollary}[theorem]{Corollary}
\newtheorem{conjecture}[theorem]{Conjecture}
\theoremstyle{definition}
\newtheorem{definition}[theorem]{Definition}

\theoremstyle{remark}
\newtheorem{remark}[theorem]{Remark}

\DeclareMathOperator{\QH}{QH}
\DeclareMathOperator{\Jac}{Jac}
\DeclareMathOperator{\Crit}{Crit}
\DeclareMathOperator{\Frac}{Frac}
\DeclareMathOperator{\Spec}{Spec}
\DeclareMathOperator{\Spf}{Spf}
\DeclareMathOperator{\Tr}{Tr}
\DeclareMathOperator{\coker}{coker}
\DeclareMathOperator{\Hess}{Hess}

\DeclareMathOperator{\im}{im}

\newcommand{\etale}{\'{e}tale}

\newcommand{\C}{\mathbb C}
\newcommand{\Z}{\mathbb Z}

\newcommand{\dbb}[1]{[\![#1]\!]}

\title[Convergence and semisimplicity for big quantum cohomology of G/P]
{Convergence and   Semisimplicity of Big Quantum Cohomology of Flag Varieties} 

\author{Changzheng Li}
\address{Changzheng Li, School of Mathematics, Sun Yat-sen University, Xingangxi Road 135, Haizhu, Guangzhou, 510275, China, Room 815}
\email{lichangzh@mail.sysu.edu.cn}
\author{Zhaoyang Liu}
\address{Zhaoyang Liu, Department of Mathematics, King’s College London, London WC2R 2LS, United Kingdom}
\email{zhaoyang.1.liu@kcl.ac.uk}

\author{Tony Yue YU}
\address{Tony Yue YU, Department of Mathematics, M/C 253-37, Caltech, 1200 E.\ California Blvd., Pasadena, CA 91125, USA}
\email{yuyuetony@gmail.com}

\author{Chi Zhang}
\address{Chi Zhang, Department of Mathematics, KIAS (Korea Institute for Advanced Study) · 85 Hoegiro, Dongdaemun-gu, Seoul 02455, Republic of Korea  }
\email{czhang@kias.re.kr}


\begin{document}

\begin{abstract}
We prove the folklore conjecture that the big quantum cohomology of a  
flag variety $G/P$ of general Lie type is convergent near the large-radius limit and is generically semisimple.
The convergence relies on a slightly modified version of Iritani's reconstruction
theorem, together with Cauchy's formula that avoids directly taking the non-equivariant limit at the origin. Generic semisimplicity
follows from a surprising application of the big quantum $D$-module mirror symmetry for
$G/P$ established by the first, third and fourth authors jointly with Hinault and Zhang.
Moreover, we prove the   regularity conjecture on the big quantum cohomology of
$G/P$ formulated by Perrin and Smirnov, which implies generic semisimplicity.
\end{abstract}

\maketitle


\section{Introduction}

The big quantum cohomology $QH^*(X)$ of  a smooth complex projective variety $X$ is a deformation of the classical cohomology ring $H^*(X)$. The structure 
   constants of the quantum product $\alpha\star\beta$ for $\alpha, \beta\in H^*(X)$    are formal power series in $\mathbb{C}\dbb{\mathbf{q}, \tau}$, defined by incorporating  genus-zero Gromov-Witten invariants of $X$.
   Convergence near the large-radius limit means that these series define
holomorphic functions on a sufficiently small polydisc centered at
$(\mathbf{q},\tau)=(\mathbf 0, \mathbf 0)$. This fundamental analytic property is
conjectured to hold in general, as already anticipated by Kontsevich
and Manin~\cite[Section~4.3]{KoMa94}.
Convergence has been shown for many examples, including 
smooth odd-dimensional complete intersections in projective spaces
by Zinger \cite{Zing} and 
smooth projective toric varieties as well as Fano manifolds with divisor-generated classical cohomology by Iritani \cite{Iri07}.

Alongside convergence, an important  question is whether $QH^*(X)$
is generically semisimple, i.e. whether its generic fiber is
a product of fields.
   This property enables the reconstruction of higher-genus
Gromov--Witten theory of $X$ from its genus-zero data, as conjectured by
Givental~\cite{Giv} and proved by Teleman~\cite{Tel}. 
Generic semisimplicity is also central to the remarkable conjecture of Dubrovin \cite{Dubrovin} for Fano manifolds $X$ and its refinements \cite{GGI, CDG}.
As clarified in~\cite{GMS}, its first part predicts that $QH^*(X)$
is generically semisimple if and only if the bounded derived category $D^b_{\mathrm{coh}}(X)$ of coherent sheaves on $X$ admits a full exceptional collection. 
Moreover, semisimplicity can also provide a criterion for convergence.
As shown by Cotti \cite{Cotti}, 
for Fano manifolds, generic semisimplicity of the small quantum cohomology  $QH^*_{\rm small}(X)=QH^*(X)|_{\tau=\mathbf 0}$ implies the convergence of  
$QH^*(X)$.
As shown by Coates and Iritani \cite{CoIr},  
the combination of genus-zero convergence and generic semisimplicity of $QH^*(X)$ implies the so-called 
NF-convergence of the descendant potentials in every genus.

In this paper we focus on flag varieties $X=G/P$, which play an important role in algebraic geometry and representation theory. Here $G$ is a simply connected complex simple Lie group and $P\subset G$ is a parabolic subgroup. These Fano manifolds need not have divisor-generated cohomology  
or  semisimple small quantum cohomology, for example, the symplectic Grassmannian $SG(2, 6)$ \cite{CMP, ChPe, GMS}.
Nevertheless, The
following is a folklore conjecture for big quantum cohomology.

\begin{conjecture}
\label{mainconj}
The big quantum cohomology of  $G/P$ is convergent near the large-radius limit and is generically  semisimple.
\end{conjecture}

\noindent Due to the aforementioned result \cite{Cotti}, the above conjecture holds whenever the small quantum cohomology $QH^*_{\rm small}(G/P)$ is generically semisimple. Hence, it  holds for complex Grassmannians \cite{SiTi, Abrams}, (co)minuscule Grassmannians \cite{CMP},  adjoint non-coadjoint Grassmannians \cite{ChPe}, a part of Grassmannians of exceptional Lie types \cite{Bel}, 
complete flag varieties $G/B$ of general Lie type  \cite
{LaRi} and partial flag varieties $SL(n, \mathbb{C})/P$ \cite{OSX}.
Generic semisimplicity of $QH^*(G/P)$ on the formal level (without assuming the convergence) has also been proved for (co)adjoint Grassmannians   ~\cite{CMPS, PeSm} and for $G/P$ with $G$  of Lie type $A$ or $C$ \cite{IrKo}.

 Combining \textbf{Theorem \ref{thm:convergence}} and \textbf{Theorem \ref{thm:analytic-semisimplicity}}, we establish the following main result.

\begin{theorem}\label{thm:main}
Conjecture \ref{mainconj} holds for any flag variety $G/P$ of general Lie type.
\end{theorem}
Combined with the existence of full exceptional collections
on $G/P$ established by Samokhin and van der Kallen
\cite[Theorem~14.5]{SaKa},   the following holds immediately. 
\begin{corollary}
    The first part of Dubrovin's conjecture holds for $G/P$.    
\end{corollary}
The proof of  the convergence part consists of   two key steps: (1) we strengthen Iritani's reconstruction theorem \cite[Proposition 5.8]{Iri07} and consider equivariant big quantum cohomology, which is divisor-generated after localizing the equivariant parameters. As a consequence, we establish convergence
of the equivariant quantum product near the large-radius limit,
uniformly along a suitable circle of specializations. (2) 
Due to the localization, a
direct passage to the non-equivariant limit at the origin is unjustified. We overcome this 
issue by applying
Cauchy's formula to the aforementioned circle of specialization, which results in  the non-equivariant convergence.

The proof of generic semisimplicity uses the big quantum $D$-module mirror theorem
established by the first, third and fourth authors jointly with Thorgal Hinault and Shaowu
Zhang~\cite[Theorem 4.38]{HLYZZ}, building on Chow's small quantum $D$-module mirror theorem~\cite{Chow}. The big mirror theorem identifies the  A-side quantum D-module with the B-side Gauss-Manin system of an unfolding  of Rietsch's mirror superpotential \cite{Rie}. On the B-side, we prove that the relative critical locus of the unfolding   is regular. Hence, we obtain the following in \textbf{Theorem \ref{thm:regularity}} by mirror symmetry, which proves  \cite[Conjecture~1.7]{PeSm}   of  Perrin and Smirnov~ on the regularity of $QH^*(G/P)$ (see Remark \ref{rmk:regular} for detailed explanations). Then we obtain the formal generic semisimplicity as a consequence. Combining it with convergence, we establish the semisimplicity part of Theorem \ref{thm:main}. 

\begin{theorem}\label{thm:formal-main}
After extending coefficients to $\C(\mathbf q)\dbb{\tau}$, the big quantum cohomology ring $\QH^*(G/P)$ is a regular ring, and its generic fiber is finite \'etale over $\operatorname{Frac}\big(\mathbb{C}(\mathbf q)\dbb{\tau}\big)$. In particular, $QH^*(G/P)$ is formally generically semisimple.
\end{theorem}

The paper is organized as follows. In Section~\ref{sec:setup} we review the background.   In Section 3, we prove the convergence by equivariant reconstruction. In Section 4, we prove the generic  semisimplicity using big quantum $D$-module mirror symmetry for flag varieties.

\subsection*{Acknowledgements}
The authors would like to thank Xiaowen Hu, Hua-Zhong Ke and Young-Hoon Kiem for helpful discussions.
C. Li is supported in part by the National Key R \&  D Program of China No.~2023YFA1009801 and Guangdong Key Laboratory of Applied Mathematics and Artificial Intelligence.
Z. Liu acknowledges support from the London School of
Geometry and Number Theory Centre for Doctoral Training (CDT), a joint venture between
University College London, Imperial College London and King's College London. T. Yu was partially
supported by NSF grants DMS-2302095 and DMS-2245099. C. Zhang is supported by a KIAS Individual Grant
(No.~MG110901) at Korea Institute for Advanced Study.

\subsection*{AI Disclosure}
ChatGPT  assisted with the proofs of the main
results at an early stage, particularly in establishing formal semisimplicity, and  assisted with the final proofreading of the manuscript.

\section{Preliminaries}\label{sec:setup} 
In this section, we review the necessary background of flag varieties. 
\subsection{Equivariant quantum cohomology of flag varieties} We refer to \cite{Hump, Brion, FuPa, Giv96} for more details of this subsection.
Let $G$ be a simply-connected complex simple Lie group of rank $n$, $B$ be a Borel subgroup of $G$, and $P$ be a  proper parabolic subgroup with $B\subset P\subset G$. 
Let $B_-$ denote the opposite Borel subgroup, so that $T\coloneqq B\cap B_-$ is a maximal torus of $G$.  
Let $\Delta=\{\alpha_1, \dots, \alpha_n\}$ be a basis of simple roots. Then $P$ corresponds to a unique subset $\Delta_P\subset \Delta$. Denote $I:=\{1, 2, \cdots, n\}$, $I_P:=\{i\in I\mid \alpha_i\in \Delta_P\}$ and $I^P:=I\setminus I_P$. 
The Weyl group $\mathsf W=\mathsf W_G$ (resp. $\mathsf W_P$)  is generated by simple reflections $s_i$, $i\in I$ (resp. $i\in I_P$).
 Let $\ell: \mathsf W\to \mathbb{Z}_{\geq 0}$ denote the standard length function and $w_0$ (resp.\ $w_P$) denote the longest element in $\mathsf W$ (resp.\ $\mathsf W_P$). 
Denote by $\mathsf W^P\subset \mathsf W$ the subset of minimal length representatives of the cosets $\mathsf W/\mathsf W_P$. 

The quotient $X=G/P$ is a Fano manifold, called a flag variety. There are Schubert varieties $X^w\coloneqq\overline{BwP/P}$ (resp.\ $X_w\coloneqq\overline{B_-wP/P}$) of (co)dimension $\ell(w)$ inside $X$, $w\in W^P$, which are all $T$-equivariant. Each $X_w$ defines a Schubert class $\sigma_w\in H^{2\ell(w)}(X, \mathbb{Z})$ by taking the Poincar\'e dual of its homology class, and also determines a $T$-equivariant Schubert class in 
$H^{2\ell(w)}_T(X)$ that is denoted as $\sigma_w$ by abuse of notation. Here and throughout, cohomology is taken with coefficients in $\mathbb{C}$ unless otherwise stated.
The ($T$-equivariant) Schubert classes form a basis of the ($T$-equivariant) cohomology:  
\[H^{\ast} (X, \mathbb{Z})=\bigoplus_{w\in \mathsf W^P} \mathbb{Z}\sigma_w,\qquad H^*_T(X)= \bigoplus_{w\in \mathsf W^P} R\sigma_w,\mbox{ where } R:=\mathbb{C}[\lambda_1, \cdots, \lambda_n]=H_T^*({\rm pt}).\]
\noindent Here we have denoted by 
 $\{-\lambda_1, \cdots, -\lambda_n\}$   the fundamental weights, to be consistent with the notation convention  in \cite{HLYZZ}.
More precisely, we view $\lambda_i$ as a character in $\mbox{Hom}(T, \mathbb{C}^{\ast})$, which induces  the one-dimensional representation $\mathbb{C}_{\lambda_i}$  of $T$.
Consider the  line bundles $\mathbb{C}_{\lambda_i}$ over a point, $i\in I$ (resp. $L_{\lambda_j}:=G\times_P\mathbb{C}_{\lambda_j}$ over $X$, $j\in I^P$). 
Then we have the equivariant first Chern classes  $\lambda_i= c_1^T(\mathbb{C}_{\lambda_i})$ and $c_1^T(L_{\lambda_j})=\sigma_{s_j}+\lambda_j$.

Let $\overline{\mathcal{M}}_{0, m}(X, d)$ denote 
the moduli space of genus-zero, $m$-pointed stable maps to $X$  of   degree $d\in H_2(X, \mathbb{Z})$
and   ${\rm ev}_i\colon \overline{\mathcal{M}}_{0, m}(X, d)\to X$ denote the $i$-th  $T$-equivariant  evaluation map. 
The natural action of $T$ on the convex space $X$ induces an action on  $\overline{\mathcal{M}}_{0,m} (X,d)$, resulting in   a $T$-equivariant  fundamental class $[\overline{\mathcal{M}}_{0,m} (X,d)]_T$.
For $\gamma_1, \dots, \gamma_m\in H^{\ast}_T(X,\mathbb{C})$, we have the genus-zero, 
$m$-point equivariant Gromov-Witten
invariant
$$  \langle \gamma_1, \dots, \gamma_m\rangle_{d}^T\coloneqq\int_{[\overline{\mathcal{M}}_{0,m}(X, d)]_T} {\rm ev}_1^*(\gamma_1)\cdot\dots \cdot {\rm ev}_m^*(\gamma_m) \in R.
$$
Taking the non-equivariant limit $\lambda=0$ gives  Gromov-Witten invariants   $\langle \gamma_1, \dots, \gamma_m\rangle_{d}\in \mathbb{C}$. 

Write $I^P=\{{i_1}, \dots, {i_r}\}$ and 
$W^P=\{v_1, \cdots, v_N\}$ with $v_j=s_{{i_j}}$ for $1\leq j\leq r$. 
For  $d\in H_2(X, \mathbb{Z})$, we have $d=\sum_j d_j[X^{s_{ {i_j}}}]$ and denote $\mathbf q^d:=\prod_{j=1}^rq_j^{d_j}$  with Novikov variables $q_j$. 
Elements of   $H^*_T(X)$ are of the form $\gamma=\sum_{i=1}^N\tau_i \sigma_{v_i}$, where  $\{\tau_i\}$ are the $R$-linear   coordinates of $H^*_T(X)$. The equivariant big quantum cohomology ring of $X$ is an $R[\mathbf q]\dbb{\tau}$-module, $\QH_T^{\ast}(X) \coloneqq (H_T^{\ast}(X)\otimes_{R}R[\mathbf q]\dbb{\tau}, \star_T) $, 
with the equivariant quantum product $\star_{T}$  defined by  
   \begin{align*}\label{eq:quantum_product}
\sigma_v \star_{T}\sigma_w &= \sum_{\eta\in W^P} \sum_{m\ge 0}\sum_{i_1,\dots,i_m} \sum_{d\in H_2(X, \mathbb{Z})}    \langle\sigma_v, \sigma_w, \sigma^\eta, \sigma_{v_{i_1}}, \cdots, \sigma_{v_{i_m}}\rangle^T_d \mathbf q^d \frac{\tau_{i_1}\cdots \tau_{i_m}}{m!}  \sigma_\eta.
\end{align*}
Here $\sigma^\eta$ denotes the equivariant Schubert class determined by $X^\eta$.
Since $X$ is Fano, the coefficient of $\tau_{i_1}\cdots\tau_{i_m}\sigma_\eta$ is a polynomial in $R[\mathbf q]$.

Taking the non-equivariant limit $\lambda=\mathbf 0$ of the equivariant quantum product gives the big quantum cohomology $QH^*(X)=(H^*(X)\otimes  \mathbb{C}[\mathbf q]\dbb{\tau}, \star)$.
By the divisor axiom of Gromov-Witten theory, for $1\leq j\leq r$, $\tau_j$ always appears in the form $q_je^{\tau_j}$ in the structure coefficients of the quantum product. We further assume $v_{r+1}={\rm id}$.
By abuse of notation, we also use $\tau$ for the $\mathbb C$-coordinates on $H^*(X)$.
Write $\tau_{\rm nd}=(\tau_{r+1},\ldots,\tau_N)$. Then the (non-equivariant) quantum product is equivalently given by 
$$\sigma_{v_i}\star \sigma_{v_j}=\sigma_{v_i}\cup \sigma_{v_j}+\sum_{l=1}^N\partial_{\tau_i}\partial_{\tau_j}\partial_{\tau_l}\Psi \sigma^{v_l},$$ where   $\Psi$ is  the (positive-degree, genus-zero) potential function   defined by    $$\Psi= \sum_{d \in H_2(X, \mathbb{Z})\setminus\{0\}}\mathbf q^d\sum_{m=0}^{\infty}\frac{1}{m!} \langle \tau,  \cdots,  \tau \rangle_{d}\in\mathbb{C}[q_1e^{\tau_1},\cdots, q_re^{\tau_r}]\dbb{\tau_{\rm nd}}.$$

\begin{definition}\label{def:convergence}
The big quantum cohomology $QH^*(X)$ is called \emph{convergent near the large-radius limit} if so are all the quantum products $\sigma_{v_i}\star\sigma_{v_j}$, or equivalently,  there exists $\epsilon>0$ such that  $\Psi$ absolutely converges as a series in $\mathbb{C}[q_1e^{\tau_1},\cdots, q_re^{\tau_r}]\dbb{\tau_{\rm nd}}$, whenever
\[
 |q_1e^{\tau_1}|<\epsilon, \cdots, |q_re^{\tau_r}|<\epsilon,\,\,\, |\tau_{r+1}|< \epsilon, \cdots, |\tau_N|<\epsilon.
 \]
 \end{definition}

\subsection{Semisimplicity}

\begin{definition}
Let $A$ be a finite-dimensional commutative unital algebra over a field $K$.
\begin{enumerate}
    \item $A$ is called  \emph{semisimple} if
$
 A\cong\prod_{\nu=1}^M K_\nu$, where   $K_\nu\supset K$ are fields with $[K_\nu:K]<\infty.$ 
 \item  $A$ is called \emph{\! finite \'{e}tale over $K$\!} if the structure morphism Spec$(A)$ $\rightarrow$ Spec$(K)$ is finite, flat and unramified.
 \item $A$ is called \emph{regular} if its localization at every prime ideal $\mathfrak{p}$ is a regular local ring $(A_{\mathfrak{p}}, \mathfrak{m}_{\mathfrak{p}})$, namely if  $\dim_{A_{\mathfrak{p}}/\mathfrak{m}_{\mathfrak{p}}}(\mathfrak{m}_{\mathfrak{p}}/\mathfrak{m}_{\mathfrak{p}}^2)$ equals the Krull dimension of $A_{\mathfrak{p}}$.
\end{enumerate}
\end{definition}

\begin{proposition}
\label{prop:finite-algebra-semisimplicity}
Let $A$ be a finite-dimensional commutative unital algebra  over a field $K$ of characteristic zero.    The following conditions are equivalent:
\begin{enumerate}
\item[\textup{(i)}] $A$ is semisimple.
\item[\textup{(ii)}] $A$ is reduced.
\item[\textup{(iii)}] $A$ is finite \'{e}tale over $K$.
\item[\textup{(iv)}] $A\otimes_K\overline K\cong\overline K^{\,N}$, where  $N=\dim_K A$ and $\overline K$ denotes  the algebraic closure of $K$.
\item[\textup{(v)}] the trace pairing
$(a,b)\mapsto\Tr_K(m_am_b)$ is nondegenerate, where $m_a$ denotes the linear operator on $A$ induced by the 
multiplication by $a$.
\item [\textup{(vi)}] $A$ is regular.
\end{enumerate}
\end{proposition}
  \begin{proof}
      The equivalence of conditions~(i)--(v) above follows from
\cite[Theorem~8.7]{AtMa} and
\cite[Propositions~I.3.1 and~I.3.2]{Miln}, using the
separability of finite field extensions in characteristic zero.
For the equivalence with~(vi), since $A$ is finite-dimensional over $K$, all its local rings
are zero-dimensional. Hence $A$ is regular if and only if
each localization $A_{\mathfrak m}$ at a maximal ideal is
a field, or equivalently, if $A$ is a finite product of
finite field extensions of $K$.
  \end{proof}
\begin{definition}
  We call the big quantum cohomology $QH^*(X)$   \emph{formally generically semisimple} without assuming convergence, if 
  $QH^*(X)\otimes_{\mathbb{C}[\mathbf{q}]\dbb{\tau}}K$ is semisimple as a $K$-algebra, where $K$ is the fraction field of ${\mathbb{C}[\mathbf{q}]\dbb{\tau}}$.

  We call $QH^*(X)$    
  \emph{(analytically) generically semisimple}, if  $QH^*(X)$ is convergent near the large-radius limit and the fiber $(H^*(X), \star_{(\mathbf q, \tau)})$ is semisimple as a $\mathbb{C}$-algebra, for a generic point $(\mathbf q, \tau)$ in the domain of convergence.
\end{definition}
\begin{remark}
For $QH^*(X)$,  our Novikov coefficient convention  differs   from
that  of \cite[Section 2.1]{PeSm}, but our 
 formal  generic semisimplicity is equivalent to the generic semisimplicity in \cite[Definition 2.3]{PeSm}, since 
 semisimplicity in characteristic zero is preserved and reflected under  field extensions.
\end{remark}
 
\section{Convergence of $QH^*(G/P)$ near the large-radius limit}
In this section, we show that the big quantum cohomology of $X=G/P$ is convergent near the large-radius limit, by considering equivariant big quantum cohomology, using Iritani's reconstruction theorem and applying Cauchy's formula.
\subsection{Uniform analytic reconstruction}

Let $H$ be an $N$-dimensional complex vector space with a distinguished vector
$1\in H$.  Fix $r\in \mathbb{Z}_{>0}$, $\delta\in \mathbb{R}_{>0}$ and write
$\mathbf q=(q_1,\dots,q_r)$ and $\tau_{\rm nd}=(\tau_{r+1},\dots,\tau_N)$. Put
\[
 \mathcal A_\delta=\{\zeta\in\C:|\zeta|=\delta\},
 \qquad
 \Delta^r_\rho(\mathbf{q})=\{\mathbf q\in\C^r:|q_a|<\rho\}.
\]

Let $\Omega_a(\zeta,\mathbf q)$, $1\le a\le r$, be $\operatorname{End}(H)$-valued
functions holomorphic on a neighborhood of $\mathcal A_\delta\times\{\mathbf q=0\}$ and
independent of $\hbar$.  Assume that for every $\zeta\in\mathcal A_\delta$ the
logarithmic connection
\begin{equation}
 \nabla^\hbar=\hbar\,d_q+\sum_{a=1}^r\Omega_a(\zeta,\mathbf q)\frac{dq_a}{q_a}
 \label{eq:small-connection}
\end{equation}
on the trivial bundle with fibre $H$ is flat, the flatness being required as an
identity in $\hbar$.  Put
\[
 p_a(\zeta)=\Omega_a(\zeta,0),\qquad
 \mathbf p(\zeta)=(p_1(\zeta),\dots,p_r(\zeta))
\].

Fix vectors $\phi_{r+1},\dots,\phi_N\in H$.  We consider formal flat extensions
of \eqref{eq:small-connection} of the form
\begin{equation}
 \widetilde\nabla^\hbar
 =\hbar\,d_{\mathbf q,\tau_{\rm nd}}+\sum_{a=1}^r\Omega_a(\zeta,\mathbf q,\tau_{\rm nd})\frac{dq_a}{q_a}
   +\sum_{j=r+1}^N\Phi_j(\zeta,\mathbf q,\tau_{\rm nd})\,d\tau_j,
 \qquad \Phi_j1=\phi_j ,
 \label{eq:big-connection}
\end{equation}
with coefficients in $\mathcal O(U)\dbb{\mathbf q,\tau_{\rm nd}}$ for some neighborhood $U$ of
$\mathcal A_\delta$, meaning coefficientwise holomorphic formal series.  Here
the flatness is relative to $(\mathbf q,\tau_{\rm nd})$ and is again required as an identity in
$\hbar$.  We write
\[
 \boldsymbol\Omega(\zeta,\mathbf q,\tau_{\rm nd})
   =(\Omega_1(\zeta,\mathbf q,\tau_{\rm nd}),\dots,\Omega_r(\zeta,\mathbf q,\tau_{\rm nd})),
 \qquad
 \boldsymbol\Omega(\zeta,\mathbf q)=\boldsymbol\Omega(\zeta, \mathbf q,0).
\]

The proposition below is Iritani's analytic reconstruction theorem
\cite[Proposition~5.8]{Iri07} with the additional assertion that
the convergence radius is uniform on the circle $\mathcal A_\delta$.

\begin{proposition}
\label{prop:uniform-reconstruction}
Assume that there are fixed polynomials $\chi_1,\dots,\chi_N$  such that
\begin{equation}
 \chi_1(\mathbf p(\zeta))1,\dots,\chi_N(\mathbf p(\zeta))1
 \label{eq:cyclic-frame}
\end{equation}
is a basis of $H$ for every $\zeta\in\mathcal A_\delta$.  Then the formal flat
extension \eqref{eq:big-connection} is unique.  Moreover, there is a constant
$\epsilon>0$, depending only on $\delta$, $\Omega$,$\chi$ and $\phi$, such that all
its matrix entries converge normally on
\[
 |q_a|<\epsilon \quad (1\leq a\leq r),\qquad |\tau_j|<\epsilon\quad (r<j\le N),
\]
uniformly for $\zeta\in\mathcal A_\delta$.
\end{proposition}

\begin{proof}
The recurrence and the elimination needed are exactly those of
\cite[Proposition~5.8]{Iri07}. The only new point here is the uniformity in $\zeta$, which we now establish by exhibiting the
input of Iritani's estimate.

Let
$
 \mathcal C(\zeta,\mathbf q)=\Bigl(\chi_1(\boldsymbol\Omega(\zeta,\mathbf q))1\ \Bigl|\ \cdots\
 \Bigr|\ \chi_N(\boldsymbol\Omega(\zeta,\mathbf q))1\Bigr).
$ 
By the hypothesis, $\det\mathcal C(\zeta,0)$ is continuous and nowhere zero on the
compact set $\mathcal A_\delta$, so $|\det\mathcal C(\zeta,0)|\ge c>0$ there. 
After shrinking the $\mathbf q$-radius, $\mathcal C$ and $\mathcal C^{-1}$ are
holomorphic and uniformly bounded on
$\mathcal A_\delta\times\overline{\Delta^r_\rho(\mathbf q)}$ for some $\rho>0$.

At $\tau_{\rm nd}=0$, it follows from the flatness of the quantum  connection that
$[\Omega_a,\Phi_j]=0$. Since $\chi_i(\boldsymbol\Omega)$
is a polynomial in $\Omega_1,\dots,\Omega_r$, it commutes with
$\Phi_j(\zeta,\mathbf{q},0)$, so column by column
\[
 \Phi_j(\zeta, \mathbf q,0)\chi_i(\boldsymbol\Omega(\zeta,\mathbf q))1
 =\chi_i(\boldsymbol\Omega(\zeta,\mathbf q))\Phi_j(\zeta, \mathbf q,0)1
 =\chi_i(\boldsymbol\Omega(\zeta,\mathbf q))\phi_j .
\]
Hence
\begin{equation}
 \Phi_j(\zeta, \mathbf q,0)
 =\bigl(\chi_1(\boldsymbol\Omega(\zeta,\mathbf q))\phi_j,\dots,
         \chi_N(\boldsymbol\Omega(\zeta,\mathbf q))\phi_j\bigr)\mathcal C(\zeta,\mathbf q)^{-1},
 \label{eq:initial-phi}
\end{equation}
which is uniformly bounded on a smaller neighborhood.

Thus all input data for the estimate in \cite[Proposition~5.8]{Iri07} are holomorphic near $\mathcal A_\delta$ and uniformly bounded there.
These data consist of the matrices $\Omega_a(\zeta,\mathbf q)$ and $\Phi_j(\zeta,\mathbf q,0)$ in a small neighborhood, together with finitely many coefficients expressing
a fixed basis of $H$.\footnote{The coefficients are denoted as $A^{(i)}_{a_1,\dots,a_l}$ in
\cite{Iri07}, which are fixed linear combinations of entries of $\mathcal C(\zeta,0)^{-1}$ and hence are holomorphic in  $\zeta$.}
The estimate itself is a purely algebraic calculation
in these data, so it yields bounds and hence gives a uniform convergence radius.
\end{proof}

\subsection{Equivariant reconstruction and convergence}
\label{sec:eq-reconstruction}
We prove the convergence of $QH^*(X)$
in Theorem~\ref{thm:convergence} through two key steps: (1) establish convergence
of the equivariant quantum product near the large-radius limit,
uniformly along a suitable circle of specializations; (2) apply
Cauchy's formula to obtain non-equivariant convergence, avoiding
an otherwise unjustified  direct passage to the non-equivariant limit.

Throughout this subsection, multiplication matrices are taken 
in the ordered basis $(\sigma_{v_1}, \cdots, \sigma_{v_N})$. Recall that   $\sigma_v$ denotes both ordinary and equivariant Schubert classes, so that $\gamma\in H^*(X)$ can naturally be viewed in $H^*_T(X)$ and $QH^*_T(X)$. Let $\mathcal{C}_\gamma^T(\lambda)$,
  $\mathcal{C}_\gamma^T(\lambda, \mathbf q, \tau)$ and $\mathcal{C}_\gamma(\mathbf q, \tau)$ be the matrices of the linear operator induced by the equivariant, equivariant quantum and quantum multiplication by $\gamma$, respectively (with $\gamma$ viewed as its canonical equivariant lifting in the first two cases). Then we have $\mathcal{C}_\gamma^T(0, \mathbf q, \tau)=\mathcal{C}_\gamma(\mathbf q, \tau)$.
 
 For $\xi\in\operatorname{Lie}(T)$, let
$\cdot|_\xi$ be the product on $H^*(X)$ obtained by evaluating
the equivariant structure coefficients in
$R=\C[\lambda]=\operatorname{Sym}(\operatorname{Lie}(T)^*)$
at $\xi$. We denote by $\mathcal{C}^T_\gamma(\lambda)|_{\xi}$ the resulting matrix of the operator  on   $(H^*(X), \cdot|_{\xi})$ given by  
 $\sigma_v\mapsto (\gamma\cdot\sigma_v)|_{\xi}$.
Let $\mathbf e_{j}$ be the $j$-th standard coordinate vector. Note that $\mathbf{e}_{r+1}$ is the coordinate vector of $\sigma_{v_{r+1}}=\sigma_{\rm id}$. Define $$D:=\sum_{i\in I^P}\sigma_{s_i}\in H^*_T(X),\qquad \mathcal V^T(\lambda)
 :=\bigl(
 \mathbf e_{r+1},\,
 \mathcal C_D^T(\lambda)\mathbf e_{r+1},\ldots,\,
 (\mathcal C_D^T(\lambda))^{N-1}\mathbf e_{r+1}
 \bigr). $$
\begin{proposition}\label{prop:cyclicity}
 There exists    $\xi_0\in\operatorname{Lie}(T)$
such that $\{D^{k}|_{\zeta\xi_0}\}_{k=0}^{N-1}$ form a $\mathbb{C}$-basis of   $(H^*(X), \cdot|_{\zeta \xi_0})$  for any $\zeta\in\C^\times$.
Moreover, we have 
\[
 \left.\mathcal V^T(\lambda)\right|_{\zeta\xi_0}
 \in
 \operatorname{Mat}_N(\C[\zeta])
 \cap
 \mathrm{GL}_N\bigl(\mathcal O(\C^\times)\bigr).
\]
\end{proposition}

\begin{proof}
By \cite[Lemma~5.11]{BCMP},   $D$
 generates   $H_T^*(X)\otimes_R \operatorname{Frac}(R)$
  (see  also \cite[Lemma 4.1.3]{CKS}).
   Consequently, $\{D^k\}_{k=0}^{N-1}$ form a $\operatorname{Frac}(R)$-basis.
Their coordinate matrix in the ordered Schubert basis is $\mathcal V^T(\lambda)$. More explicitly, the $(i,j)$-entry of $\mathcal V^T(\lambda)$
is the coefficient of $\sigma_{v_i}$ in $D^{j-1}$.
If nonzero, this entry is a homogeneous polynomial in $R$
of polynomial degree $j-1-\ell(v_i)$.
It follows that  
$ \mathcal V^T(\lambda)\in\operatorname{Mat}_N(R)$ and 
$\det\mathcal V^T(\lambda)$ is a nonzero  homogeneous polynomial in $R$ of polynomial degree
$m:={N(N-1)\over 2}-\sum_{v\in W^P}\ell(v)\geq 0$, whose zero locus is either a hypersurface or the empty set in $\operatorname{Lie}(T)$. Take any element   $\xi_0$ in the complement of this zero locus. Then  for any  $\zeta\ne0$, we have  
\[
 \det\left(
   \left.\mathcal V^T(\lambda)\right|_{\zeta\xi_0}
 \right)
 =\zeta^m \det\mathcal V^T(\lambda)|_{\xi_0}\ne0.
\]
By the adjugate formula, the inverse of 
$\mathcal V^T(\lambda)|_{\zeta\xi_0}$  has entries in
$\C[\zeta,\zeta^{-1}]$.
\end{proof}

Denote $x=(x_1,\cdots, x_N):=(q_1e^{\tau_1},\ldots,q_re^{\tau_r}, \tau_{r+1},\ldots,\tau_N)$.
Recall $\mathcal{A}_\delta=\{\zeta\in \mathbb{C}\mid|\zeta|=\delta\}$.
\begin{proposition}
\label{prop:eq-convergence}
Let $\xi_0$ be as in Proposition~\ref{prop:cyclicity} and fix $\delta>0$.  Then
there is $\epsilon>0$ such that, for every $\gamma\in H^*(X)$, every matrix
entry of $\mathcal{C}^T_\gamma(\lambda,q,\tau)\big|_{\zeta\xi_0}$ is a series
\[
 \sum_{J\in\Z_{\ge0}^N}f_J(\zeta)\,x^J,\qquad f_J\in\C[\zeta],
\]
converging normally on the polydisc $\Delta^N_\epsilon(x)=\{x\in\C^N \mid |x_j|<\epsilon\}$,
uniformly for $\zeta\in\mathcal A_\delta$, i.e.
\[
 \sum_{J\in\Z_{\ge0}^N}\ \sup_{\substack{\zeta\in\mathcal A_\delta\\ x\in K}}
 |f_J(\zeta)x^J|<\infty
 \qquad\text{for every compact }K\Subset\Delta^N_\epsilon(x) .
\]
\end{proposition}

\begin{proof}
After specializing the equivariant parameters at
$\lambda=\zeta\xi_0$, we apply Proposition~\ref{prop:uniform-reconstruction} to the equivariant small
quantum connection $\nabla^{\hbar}$ (without derivation along the $\hbar\partial_{\hbar}$-direction). Precisely, take $H=H^*(X)$ and for $1\le a\le r$, 
\[
 \Omega_a(\zeta,\mathbf q):=\mathcal C^T_{\sigma_{s_{i_a}}}(\zeta\xi_0,q,0)\in\operatorname{End}(H),
\]
so that $\nabla^\hbar=\hbar\,d_q+\sum_{a=1}^r\Omega_a(\zeta,\mathbf q)dq_a/q_a$. Below we check the hypotheses of Proposition~\ref{prop:uniform-reconstruction}.

Since $X$ is Fano, the multiplication matrices
$\mathcal C_{\sigma_{s_{i_a}}}^T(\lambda,\mathbf q,0)$ have entries
in $R[\mathbf q]$ (see \cite[Theorem 6.4]{Mihalcea} for a precise formula). Their specializations are therefore holomorphic
near $\mathcal{A}_\delta\times\{\mathbf q=0\}$. Flatness is given by the WDVV equations. 
Put $\chi_k(z_1,\dots,z_r):=(\sum_{a=1}^rz_a)^{k-1}$ for
$1\le k\le N$.  Since $p_a(\zeta)=\Omega_a(\zeta,0)$ is given by the equivariant cup
product by $\sigma_{s_{i_a}}$ with specialization at $\zeta\xi_0$, the vectors
$\chi_k(\mathbf p(\zeta))1$ are precisely the columns of
$\mathcal{V}^T(\lambda)\big|_{\zeta\xi_0}$.
By Proposition~\ref{prop:cyclicity}, this frame and its inverse are
holomorphic near $\mathcal A_\delta$.

By the divisor axiom, $\tau_1,\cdots,\tau_r$ enter only through $q_ae^{\tau_a}$. So we can first assume $\tau_1=\cdots=\tau_r=0$. The formal equivariant quantum
connection is a flat extension of the form \eqref{eq:big-connection}
in the coordinates $(\tau_{r+1},\cdots,\tau_N)$ with $\Phi_j$ the matrix of the quantum multiplication by $\sigma_{v_j}$. By the fundamental class axiom,
$\mathcal C_{\sigma_{v_j}}^T\mathbf e_{r+1}=\mathbf e_j$ for $j>r$.
Apply Proposition~\ref{prop:uniform-reconstruction} with the constant
vectors $\mathbf e_{r+1},\ldots,\mathbf e_N$.
It follows that the matrices $\Omega_a(\zeta, \mathbf q,\tau)$ and $\Phi_j(\zeta, \mathbf q,\tau)$
converge normally on a common polydisc
$\{|q_a|<\epsilon,\ |\tau_j|<\epsilon\}$, where $\epsilon$ depends only on
$\delta,\Omega,\chi$ and $\phi$, in particular not on $\zeta$. 

Write
$\gamma=\sum_{i}\gamma^i\sigma_{v_i}$.  By definition,
\[
 \mathcal{C}^T_\gamma(\lambda, \mathbf q,\tau)\big|_{\zeta\xi_0}
 =\sum_{i=1}^r\gamma^i\Omega_i(\zeta,\mathbf q,\tau)
 +\gamma^{r+1}\cdot\mathrm{id}
 +\sum_{j=r+2}^N\gamma^j\Phi_j(\zeta,\mathbf q,\tau),
\]
so every $\mathcal{C}^T_\gamma$ converges on the same polydisc, with a radius independent of
$\gamma$.  Finally each $f_J(\zeta)$ is a polynomial, because
$\mathcal{C}^T_\gamma(\lambda,\mathbf q,\tau)$ has coefficients in $R[q]\dbb{\tau}$.

It remains to add back the divisor directions.  Indeed, we have
\begin{equation*}
\mathcal{C}^T_\gamma(\lambda, \mathbf q,\tau)
 =\mathcal{C}^T_\gamma\bigl(\lambda,(q_ae^{\tau_a})_{a=1}^r,\tau_{r+1},\cdots,\tau_N\bigr),
 \label{eq:divisor-axiom}
\end{equation*}
so in the coordinates $x$ the common polydisc is exactly
$\Delta^N_\epsilon(x)$.  Thus the statement holds.
\end{proof}

\begin{theorem}\label{thm:convergence}
The big quantum cohomology
$QH^*(X)$ converges near the large-radius limit.
\end{theorem}

\begin{proof}
  Fix $\delta>0$ and let
$
 \Delta_\epsilon
 :=\{x\in\mathbb C^N\mid |x_j|<\epsilon,\ 1\le j\le N\}
$ 
be the polydisc supplied by
Proposition~\ref{prop:eq-convergence}.
Each entry of
$\left.\mathcal C_\gamma^T(\lambda,\mathbf q,\tau)
\right|_{\zeta\xi_0}$
has an expansion $\sum_J f_J(\zeta)x^J$, with
$f_J(\zeta)\in\mathbb C[\zeta]$.
By Cauchy's formula, we have
\[
 f_J(0)=\frac{1}{2\pi \sqrt{-1}}\int_{|\zeta|=\delta}
 f_J(\zeta)\frac{d\zeta}{\zeta},
 \qquad
 |f_J(0)|\le\sup_{|\zeta|=\delta}|f_J(\zeta)|.
\]
Hence, the uniform normal convergence of $\sum_J f_J(\zeta)x^J$  implies the normal
convergence of $\sum_J f_J(0)x^J$ on $\Delta_\epsilon$
and permits termwise integration along the circle $|\zeta|=\delta$. Since evaluation at
$\lambda=0$ recovers the non-equivariant structure coefficients,
we obtain
\[
 \mathcal C_\gamma(\mathbf q,\tau)
 =\mathcal C_\gamma^T(0,\mathbf q,\tau)
 =\frac{1}{2\pi \sqrt{-1}}\int_{|\zeta|=\delta}
 \left.\mathcal C_\gamma^T(\lambda,\mathbf q,\tau)
 \right|_{\zeta\xi_0}\frac{d\zeta}{\zeta}.
\]
Thus all non-equivariant quantum multiplication matrices
converge normally on $\Delta_\epsilon$.
By  
Definition~\ref{def:convergence}, the statement follows.
\end{proof}

\section{Generic semisimplicity of $QH^*(G/P)$}
In this section, we prove the generic semisimplicity by using the big quantum $D$-module mirror symmetry for $X=G/P$ established in \cite{HLYZZ}. By the big mirror theorem therein,   
algebraic structures of big quantum cohomology are equivalent to the corresponding geometric properties of the critical points of the superpotential, which are easier to describe.  

\subsection{The mirror and its deformation}\label{sec:formal-mirror} 
We begin with a brief review  needed to explain the big mirror theorem \cite[Theorem 4.38]{HLYZZ}, and refer to \cite{Rie, HLYZZ} and \cite[Section 4]{Chow} for more details on the relevant background. 

On the A-side, we enlarge the coefficients by letting $\QH(X)=\QH^*(X)\otimes_{\C[\mathbf q]\dbb{\tau}}S_X$, where 
$$k:=\C(\mathbf q), \qquad S_X=k\dbb{\tau}=k\dbb{\tau_1,\ldots,\tau_N}.$$ 
The quantum module
and its connection are respectively given by
\[
 \mathcal H_A=H^*(X)\otimes_\C S_X\dbb{\hbar},\qquad
 \nabla^A_{\partial_{\tau_i}}
 =\partial_{\tau_i}+\hbar^{-1}(\sigma_{v_i}\star).
\] 

On the B-side, we let $U_P$ denote the specific projected open Richardson variety inside the Langlands dual flag variety $G^\vee/P^\vee$, given by the complement of a distinguished anti-canonical divisor of $G^\vee/P^\vee$. Let $Y=U_P\times (\C^\times)^r$ with coordinate $(\mathbf z, \mathbf q)=(z_1, \cdots, z_M, q_1, \cdots, q_r)$. 
The B-model mirror to  $X$ is given by Rietsch's superpotential $W=W(\mathbf z, \mathbf q): Y\to \C$,  a regular function on $Y$. Below we work with $\mathcal O(Y)\otimes_{\mathbb C[q_1^{\pm1},\ldots,q_r^{\pm1}]}k$ and denote this algebra by $\mathcal O(Y)$ by abuse of notation.
The (relative) Jacobi ring is given by 
\[
 \Jac(W)=\mathcal O(Y)/(\partial_{z_1} W, \cdots, \partial_{z_M}W).
\]
It defines the (relative) critical scheme $\Crit(W)$ in $Y$, which is possibly non-reduced.
By
\cite[Theorem~1.2(4)]{Chow}, we have a ring isomorphism  
\begin{equation}\label{eq:small-ring-mirror}
 \Jac(W)\cong\QH^*_{\mathrm{small}}(X)\otimes_{\C[q]}k.
\end{equation}
 
 To resolve the possibly non-reduced issue,  we consider  the mirror symmetry for $X$ on the level of big quantum $D$-modules. Similar to the A-side coordinates, we introduce $\mathbf y=(y_1,\ldots,y_N)$ and put $S=k\dbb{\mathbf y}$ on the B-side. 
Choose functions $f_j\in\mathcal O(Y)$, for $r+1\le j\le N$, whose classes
correspond to $\sigma_{v_j}$ under \eqref{eq:small-ring-mirror}, following
\cite[Construction~4.28]{HLYZZ}.
The divisor axiom prescribes the substitution $q_a\mapsto q_ae^{y_a}$
for $1\le a\le r$, which we abbreviate as $\mathbf q\mapsto \mathbf qe^{\mathbf y}$.
  We consider the  unfolding $F\in \mathcal{O}(Y)\dbb{\mathbf y}$,
\begin{equation}\label{eq:full-potential}
 F=W(\mathbf{z}, \mathbf{q}e^{\mathbf y}) 
       +\sum_{j=r+1}^N y_j f_j.
\end{equation}
 The (relative) Jacobian ring of $F$ is given by
\[
 \mathcal A=\Jac(F):=\mathcal O(Y)\dbb{\mathbf y}/
 \bigl(\partial_{z_1}F, \cdots, \partial_{z_M}F\bigr).
\]

Denote by $d_U$ the differential over $k$. Following
\cite[Section~4.3.2]{HLYZZ}, we use the non-equivariant formal Brieskorn lattice

\[
 \mathcal H_B
 =
 \frac{\Omega^M_{Y/k}[[\mathbf y,\hbar]]}
 {(\hbar d_U+d_UF\wedge)
  \Omega^{M-1}_{Y/k}[[\mathbf y,\hbar]]}.
\]
Here $\Omega^j_{Y/k}$ denotes the space of algebraic differential
$j$-forms on $Y$ over $k$. The mirror space $Y$ has a
nowhere-vanishing volume form $\omega_Y$. By the standard reduction
of a Brieskorn lattice modulo its parameter
(see \cite[Proposition~2.13]{DouSab},
see also \cite[Lemma~4.10]{Chow}), we have
\begin{equation}\label{eq:B-mod-hbar}
 \mathcal H_B/\hbar\mathcal H_B \cong \mathcal A\omega_Y.
\end{equation}
This follows directly by reducing
the defining cokernel modulo $\hbar$ and using $\omega_Y$ to
identify top forms with functions. Since cokernels commute with
base change, this reduction is compatible with the non-equivariant specialization
used in \cite[Section~4.3.2]{HLYZZ}.

The Gauss--Manin connection describes how classes of forms vary
with the parameters $\mathbf y$. In the product presentation,
\[
 \hbar\nabla^B_{\partial_{y_i}}[\omega]
 =
 [\hbar\partial_{y_i}\omega+(\partial_{y_i}F)\omega].
\]
We state the completed
non-equivariant big mirror theorem
in this notation.

\begin{proposition}[Non-equivariant big mirror symmetry
{\cite[Theorem~4.38]{HLYZZ}}]
\label{prop:big-mirror-input}
There is an isomorphism of formal spectra
$\mu:\Spf S\xrightarrow{\sim}\Spf S_X$, called the mirror map,
together with an isomorphism of modules with connection
\[
 \Phi:(\mathcal H_B,\nabla^B)
 \xrightarrow{\sim}\mu^*(\mathcal H_A,\nabla^A).
\]
In particular, $\mathcal H_B$ is free of rank $N$ over $S\dbb{\hbar}$.

Moreover, the induced  $S$-module isomorphism
$
 \Phi_0:\mathcal H_B/\hbar\mathcal H_B \xrightarrow{\sim}\mu^*\QH(X)
$ 
satisfies 
\begin{equation}  \label{eqn:compt}   \Phi_0\bigl([\partial_{y_i}F]\eta\bigr)
 =
 \sum_{j=1}^N
 \frac{\partial\mu_j(\mathbf y)}{\partial y_i}\mu^*(\sigma_{v_j})
 \star\Phi_0(\eta),
 \qquad \forall \eta\in \mathcal H_B/\hbar\mathcal H_B,\qquad  1\leq i\leq N,
\end{equation}
where    
  $\star$ is understood as  the pulled-back quantum product.
\end{proposition}

\subsection{Regularity of $\mathcal A$}
By the small mirror isomorphism \eqref{eq:small-ring-mirror} and our choice of $f_i$, the classes
\begin{equation}\label{eq:central-KS}
\left[\left.\partial_{y_i}F\right|_{\mathbf y=0}\right]
=\begin{cases}
[q_i\partial_{q_i}W],&1\leq i\leq r,\\
[1],&i=r+1,\\
[f_i],&r+2\leq i\leq N
\end{cases}
\quad\longleftrightarrow\quad\sigma_{v_i}
\end{equation}
  form a basis of $\Jac(W)$. By Proposition~4.1
and (4.3), $\mathcal A$ is free of rank $N$ over $S$.
Thus by Nakayama's lemma,    the classes
$[\partial_{y_1}F],\ldots,[\partial_{y_N}F]$ form an $S$-basis of $\mathcal A$.

Given $p\in Y$, we consider the residue field $\kappa(p)=\mathcal{O}(Y)/\mathfrak{m}_p$.
Let $T_pY$ and $T_p^*Y$ denote the tangent
and cotangent spaces of $\Spec\mathcal O(Y)$ over $k$ at $p$. 
Consider the  Hessian map 
$$\operatorname{Hess}_p(W): T_pY\to T_p^*Y;\quad
\sum_{i=1}^M a_i\partial_{z_i}
\longmapsto
\sum_{j=1}^M
\left(
\sum_{i=1}^M
a_i\frac{\partial^2W}{\partial z_i\partial z_j}(p)
\right)dz_j.
$$ 
\begin{lemma}\label{lem:hessian-cokernel}
For any  $p\in\Crit(W)$,  the following map is well-defined and is  surjective.
\begin{equation}\label{eq:hessian-cokernel}
\Jac(W)\otimes_k\kappa(p)\longrightarrow
\coker\bigl(\Hess_p(W)\bigr),
\qquad [h]\otimes c\longmapsto\overline{c\,d_Uh(p)}.
\end{equation}
\end{lemma}
\begin{proof}
 Since $p\in \operatorname{Crit}(W)$, an element
$h=\sum_i a_i\partial_{z_i}W$ of the Jacobian ideal satisfies
\[
d_Uh(p)=\sum_i a_i(p)\,d_U(\partial_{z_i}W)(p)
\in\im\Hess_p(W).
\]
Thus the map is well defined. It is surjective because differentials of
elements of $\mathcal O(Y)$ span
$T_p^*Y=\Omega^1_{Y/k}\otimes_{\mathcal O(Y)}\kappa(p)$.
\end{proof}

\begin{proposition}\label{prop:transversality}
The $S$-algebra $\mathcal A$ is a regular ring, finite free of rank $N$.
Its generic fiber
$\mathcal A\otimes_S\Frac(S)$ over $S$ is finite \etale.
\end{proposition}
\begin{proof}
Since $\mathcal A$ is finite over $S$, every maximal ideal $\mathfrak m$ of $\mathcal A$
lies over $(y_1,\ldots,y_N)\subset S$ and corresponds to a closed point
$p\in\Crit(W)$.
Let $\widehat{\mathcal O}_{Y,p}$ denote the completion of
$\mathcal O_{Y,p}$ at its maximal ideal.
Since $\Spec\mathcal O(Y)$ is smooth of dimension $M$ over $k$,
we have,  
$ 
\widehat{\mathcal O}_{Y,p}
\cong\kappa(p)\dbb{u_1,\ldots,u_M}.
$ 
The completed ambient local ring at $(p,0)$ is
\[
B_p:=\widehat{\mathcal O}_{Y,p}\dbb{\mathbf y}
\cong
\kappa(p)\dbb{u_1,\ldots,u_M,y_1,\ldots,y_N},
\]
which is regular of dimension $M+N$ with  unique maximal ideal $\mathfrak n_p=(u_1, \cdots, u_M, y_1, \cdots, y_N)$.

By Lemma~\ref{lem:hessian-cokernel} and \eqref{eq:central-KS}, the composition map  
\begin{equation*}\label{eq:critical-linearization}
\begin{aligned}
T_pY\oplus\kappa(p)^N&\overset{\mathcal{L}_p}{\longrightarrow} T_p^*Y\longrightarrow \operatorname{coker}(\operatorname{Hess}_p(W))
\end{aligned}
\end{equation*}
is surjective, where  $ \mathcal{L}_p(v,c):=\Hess_p(W)(v)
+\sum_{i=1}^N c_i\,d_U\!\left(
\left.\partial_{y_i}F\right|_{\mathbf y=0}\right)(p).$ Since
$\operatorname{im}\Hess_p(W)
=\mathcal L_p(T_pY\oplus0)$,
it follows that $\mathcal L_p$ is surjective as well.  Consequently, we have the injection
\[
\mathcal L_p^\vee:
(T_p^*Y)^\vee=T_pY
\longrightarrow
(T_pY\oplus\kappa(p)^N)^\vee
\cong\mathfrak n_p/\mathfrak n_p^2
\qquad
\partial_{z_i}|_{p}\longmapsto
\partial_{z_i}F\bmod\mathfrak n_p^2.
\]
Here $\partial_{z_i}F\in \mathcal{O}(Y)\dbb{y}$ and vanishes at $(p, 0)$, and hence lies in $\mathfrak n_p$.
Since the vectors $\partial_{z_i}|_p$
form a basis of $T_pY$, the classes
$\partial_{z_i}F\bmod\mathfrak n_p^2
$ 
are linearly independent over $\kappa(p)$. Thus by Nakayama's lemma, 
$\partial_{z_1}F,\ldots,\partial_{z_M}F$ form part of   a
regular system  of parameters 
of $B_p$.
 Since $B_p$ is regular, 
the completion $ 
\widehat{\mathcal A_{\mathfrak m}}
\cong
B_p/(\partial_{z_1}F,\ldots,\partial_{z_M}F)$ 
is regular of dimension $(M+N)-M=N$.
Hence, 
$\mathcal A_{\mathfrak m}$ is regular, with
$ 
\dim\mathcal A_{\mathfrak m}
=\dim\widehat{\mathcal A_{\mathfrak m}}=N.
$ 
Since $\mathfrak m$ was arbitrary, $\mathcal A$ is regular.

The generic fiber $\mathcal A\otimes_S\Frac(S)$ is regular,
being a localization of $\mathcal A$ and finite-dimensional
over $\Frac(S)$. It is therefore finite \etale\ over $\Frac(S)$
by Proposition~\ref{prop:finite-algebra-semisimplicity}.   
\end{proof}

\subsection{The quantum spectral cover}
The quantum spectral cover is $\Spec\QH(X)\to\Spec S_X$.
We identify its coordinate ring, after the mirror change of variables, with
$\mathcal A$.

\begin{theorem}\label{thm:regularity}
The $S_X$-algebra $\QH(X)$ is a regular ring and is finite free of rank $N$.
Its generic fiber is finite \etale\ over $\Frac(S_X)$.
Moreover, $\QH^*(X)$ is formally generically semisimple in the sense of
Definition~2.4.
\end{theorem}
\begin{proof}
Note that $\mathcal A\omega_Y\cong
\mathcal H_B/\hbar\mathcal H_B$ and that the classes
$[\partial_{y_i}F]$ form an $S$-basis of $\mathcal A$.
Since $\mu$ is an isomorphism, the multiplication
compatibility in \eqref{eqn:compt} identifies, via $\Phi_0$, the faithful multiplication
representations of $\mathcal A$ and $\mu^*\QH(X)$.
Thus the following map
\begin{equation}\label{eq:big-ring-mirror}
\mathcal A\xrightarrow{\sim}\mathcal \mu^*\QH(X),
\qquad
[\partial_{y_i}F]\longmapsto
\sum_{j=1}^N
\frac{\partial\mu_j}{\partial y_i}\mu^*(\sigma_{v_j})
\end{equation} is    an isomorphism of $S$-algebras. 
Hence, the first two assertions
follow from Proposition~\ref{prop:transversality}.

For the last assertion, consider the trace discriminant 
\begin{equation}\label{eq:trace-discriminant}
\mathcal D(\mathbf q,\boldsymbol\tau)=
\det\bigl(\Tr(\mathcal C_{\sigma_{v_i}}(\mathbf q,\boldsymbol\tau)
\mathcal C_{\sigma_{v_j}}(\mathbf q,\boldsymbol\tau))\bigr)_{1\leq i,j\leq N}
\in\C[\mathbf q]\dbb{\boldsymbol\tau},
\end{equation}
where $\mathcal C_\gamma$ is the multiplication matrix from Section~3.2.
Since the generic fiber is finite \'etale,   the image of
$\mathcal D$ in $\Frac(S_X)$ is nonzero.
Since the coefficientwise inclusion
$\C[\mathbf q]\dbb{\boldsymbol\tau}\hookrightarrow S_X$ is injective,
$\mathcal D$ is nonzero in the original coefficient ring as well.
Hence, the last assertion follows from   Proposition~2.3.
\end{proof}

\begin{theorem}\label{thm:analytic-semisimplicity}
On every sufficiently small large-radius polydisc, the semisimple locus of the
analytic big quantum cohomology of $X$ is open and dense.
\end{theorem}
\begin{proof}
Use the coordinates $\mathbf x=(q_1e^{\tau_1},\ldots,q_re^{\tau_r},
\tau_{r+1},\ldots,\tau_N)$ of Section~3.2 and put
\[
\mathcal D_0(\mathbf x)=\mathcal D(\mathbf q,\boldsymbol\tau)
\big|_{q_a=x_a,\,\tau_a=0\ (1\leq a\leq r),\,\tau_j=x_j\ (r<j\leq N)}.
\]
The divisor axiom gives
$ 
\mathcal D(\mathbf q,\boldsymbol\tau)
=\mathcal D_0(q_1e^{\tau_1},\ldots,q_re^{\tau_r},\tau_{r+1},\ldots,\tau_N).
$ 
Hence $\mathcal D_0$ is nonzero by Theorem~\ref{thm:regularity}.
It is holomorphic on a sufficiently small large-radius polydisc by
Theorem~3.4. By Proposition~2.3, the semisimple locus is exactly
$\{\mathcal D_0\ne0\}$. This is open and dense because $\mathcal D_0$
is a nonzero holomorphic function.
\end{proof}

\begin{corollary}\label{cor:semisimple-groups}
Let $\widetilde G$ be a connected complex semisimple Lie group and
$\widetilde P\subset\widetilde G$ a parabolic subgroup.
The big quantum cohomology of $\widetilde G/\widetilde P$ converges near the
large-radius limit and is generically semisimple, both formally and analytically.
\end{corollary}
\begin{proof}
After passing to the simply connected cover, $\widetilde G/\widetilde P$ is a
product of flag varieties for simple groups. The quantum product formula
\cite{Behr,KoMa} and Kaufmann's convergence theorem
\cite[Theorem~2.2.1]{Kauf} give convergence for the product.
On the locus of parameters pulled back from the factors, its quantum algebra
is the tensor product of the factor quantum algebras. Choose semisimple points
for the factors, arbitrarily near their large-radius limits, using
Theorem~\ref{thm:analytic-semisimplicity}. Their tensor product is semisimple,
so the trace discriminant of the product is not identically zero.
Its nonvanishing locus gives analytic generic semisimplicity, and its nonzero
Taylor series gives formal generic semisimplicity by Proposition~2.3.
\end{proof}

\begin{remark}\label{rmk:regular}
In the regularity conjecture  of Perrin and Smirnov \cite[Conjecture~1.7]{PeSm}, the big quantum cohomology was defined over $\overline{\mathbb{Q}(\!(\mathbf q)\!)}\dbb{\tau}$.
Regularity
of finite algebras over a formal power series ring is
preserved and reflected under completed extension of
characteristic-zero coefficient fields.
Hence, by Theorem~\ref{thm:regularity}, the big quantum cohomology of $G/P$ is also regular with respect to the coefficient convention in \cite{PeSm}.
 
  As explained in the proof of
Corollary~\ref{cor:semisimple-groups},
$\widetilde G/\widetilde P$ is a product of flag varieties
for simple groups.
By Theorem~\ref{thm:regularity}, the big quantum cohomology
of each factor is regular.
By the completed tensor-product description in
\cite[proof of Proposition~3.3.5]{Kauf},
the restriction to parameters pulled back from the factors
is regular.
Then the regularity of $QH^*(\tilde G/\tilde P)$, after extending coefficients to $\C(\mathbf q)\dbb{\tau}$, follows from its flatness over the power series ring in the remaining  parameters.
\end{remark}

\end{document}